\documentclass[12pt]{article}

\usepackage{amsmath}
\usepackage{amsfonts}
\usepackage{amssymb}
\usepackage{amscd}
\usepackage{amsthm}
\usepackage{bbm}
\usepackage{color}      
\usepackage[linktocpage=true,plainpages=false,pdfpagelabels=false]{hyperref}

\input xypic

\usepackage[matrix,arrow,curve]{xy}

\usepackage{color}


\newtheorem{prop}{Proposition}
\newtheorem{nt}{Remark}
\newtheorem{Th}{Theorem}
\newtheorem{lemma}{Lemma}

\newtheorem{example}{Example}

\newfont{\ssdbl}{msbm8}
\newfont{\sdbl}{msbm9}
\newfont{\dbl}{msbm10 at 12pt}
\newcommand{\eqdef}{\stackrel{\rm def}{=}}

\newcommand{\oo}{{\cal O}}
\newcommand{\ff}{{\cal F}}

\newcommand{\Spec}{\mathop {\rm Spec}}

\newcommand{\Frac}{\mathop {\rm Frac}}

\newcommand{\da}{\mathbb{A}}
\newcommand{\dz}{\mathbb{Z}}

\newcommand{\dr}{\mathbb{R}}

\newcommand{\Z}{\dz}
\newcommand{\sdz}{\dz}

\newcommand{\Ker}{{\rm Ker}\:}
\newcommand{\Image}{{\rm Im}\:}

\newcommand{\lrto}{\longrightarrow}

\newcommand{\df}{\mathbb{F}}

\newcommand{\cF}{\hat {\cal F}}

\def\Z{{\mathbb Z}}
\def\Q{{\mathbb Q}}

\newcommand{\nsubset}{\not{\subset}}

\newcommand{\pro}{\mathop{\rm pr}}
\newcommand{\mh}{{\mathcal H}}

\begin{document}

\author{
D. V. Osipov
}

\title{On adelic quotient group for algebraic surface
\footnotetext{This work is supported by the Russian Science Foundation  under  grant 14-50-00005.}
}
\date{}

\maketitle

\begin{abstract}
We calculate explicitly an adelic quotient group for an excellent Noetherian  normal integral  two-dimensional separated scheme. An application to an irreducible normal projective algebraic surface over a field is given.
\end{abstract}

\section{Introduction}
Let $K$ be a global field or the field of rational functions of an algebraic curve over a field  $k$, and $\da_K$ be the group of adeles of the field $K$.  Then it is well-known that $\da_K/K$ is a compact topological space in the number theory case and  a  linearly compact $k$-vector space in  the geometric case, see, e.g.,~\cite{C} and Section~\ref{Sec1} below. Moreover, the strong approximation theorem implies the following exact sequence, which we write in the simplest case  $K= \Q$:
\begin{equation}  \label{quot_Q}
0 \lrto \widehat{\dz}  \lrto \da_{\Q}/ \Q  \lrto \dr/\dz  \lrto 0  \, \mbox{.}
\end{equation}

The goal of this note is to calculate explicitly, similarly to formula~\eqref{quot_Q}, an adelic quotient group
\begin{equation}  \label{ad_quot}
\da_X(\ff) / \left( \da_{X, 01}(\ff)  + \da_{X, 02}(\ff)  \right)  \, \mbox{,}
\end{equation}
where $X$ is an
excellent Noetherian (e.g., finite type over $k$ or over $\Z$) normal integral  two-dimensional separated scheme $X$,
$\da_X(\ff)$ is the group of higher adeles of a locally free sheaf
$\ff$ on $X$, the subgroup $\da_{X, 01}(\ff)  + \da_{X, 02}(\ff)$ is an analog of the above subgroup $K \subset \da_K$. Instead of fixing a valuation as, for example, the  Archimedean valuation in formula~\eqref{quot_Q} (or a point on a connected projective curve)
we fix a reduced one-dimensional closed equidimensional subscheme $C$
of the scheme~$X$ such that the open subscheme $X \setminus C$ is affine.

Higher adeles were introduced by A.~N.~Parshin in~\cite{P3}  for the case of a smooth projective surface over a field, and by A.~A.~Beilinson in a short note~\cite{B}
   (which did not contain proofs) for the case of arbitrary Noetherian schemes. Later the proofs of Beilinson's results on higher adeles appeared in~\cite{H}. A survey of higher adeles is contained also in~\cite{Osi}.

The main goal  in the higher adeles program is the possible arithmetic applications for the study of  zeta- and   $L$-functions  of  arithmetic  schemes, see the excellent survey of A.~N.~Parshin~\cite{P4}.

Recently there were also found  the interesting  connections and  interactions with many questions  between the  higher adeles  theory  and the Langlands program, see~\cite{P5}.

As an application, we deduce from our calculations of adelic quotient group~\eqref{ad_quot} the corresponding quotient group when $X$ is a projective irreducible normal surface over a field $k$, and $C$ is the support of an ample divisor. As in the one-dimensional case, this quotient group will be a linearly compact $k$-vector space.
We note that this quotient group when the surface $X$ is smooth and the sheaf $\ff = \oo_X$ was calculated in~\cite[\S~14]{OsipPar2}. But there were the gaps in the proof of Theorem~3 from~\cite{OsipPar2}, see Remark~\ref{gap} below.

Another application of the calculation of adelic quotient group~\eqref{ad_quot} for two-dimensional schemes will be given in the subsequent paper~\cite{Osip1}.
It will be  calculated  explicitly an
adelic quotient group on an arithmetic surface when the fibres over Archimedean points are taken into account.

This note is organized as follows. In Section~\ref{Sec1} we recall the one-dimensional case.
In Section~\ref{ad_form} we recall the basic notation for adeles on two-dimensional normal excellent integral separated Noetherian schemes.
In Section~\ref{sur_map} we construct a surjective map from the adelic quotient group~\ref{ad_quot}.
In Section~\ref{calc_ker} we calculate the kernel of this map.  We obtain in this section  Theorem~\ref{ex-seq} with calculation of adelic quotient group~\eqref{ad_quot}. In Section~\ref{proj-sur} we apply above calculations to the case of projective normal algebraic surface, see Theorem~\ref{anal-II}.

I am grateful to A.~N.~Parshin, because our joint paper~\cite{OsipPar2} led to this note. I am grateful to A.~B.~Zheglov, who provided me the reference~\cite{BD}.

\section{One-dimensional case}  \label{Sec1}

To better understand the two-dimensional case we first recall the one-dimensional case.

Let $D$ be an irreducible   algebraic curve over a field $k$. Let $\eta $ be the generic point of $D$,
and $\ff$ be  a locally free sheaf of finite rank on   $D$. We fix a  (closed) point $p$ on the curve $D$.

For any    point $q$ on $D$ let $\hat{\oo}_q$ be the  completion of the local ring $\oo_q$ of the point $q$,  $\cF_q$ be the completion of the stalk $\ff$ at $q$, and $K_q$ be the localization of the ring $\hat{\oo}_q$ with respect to the multiplicative system $\oo_q \setminus 0$.

Let $j : U=D \setminus p \hookrightarrow D$ be an open embedding. We consider a subgroup
$$
\xymatrix{
A_U(\ff) = H^0 (U, j^* \ff )  \; \ar@{^{(}->}[r]   & \; \ff_{\eta}  \, \mbox{,}
}
$$
and the usual adelic product
$$\da_D(\ff) = \mathop{{\prod}'}_{q \in D} \cF_q  \otimes_{\hat{\oo}_q }  K_q  \mbox{.}$$
over all (closed) points of $D$.

Our goal is to construct an exact sequence:
\begin{equation}  \label{ex-one-adeles}
0 \lrto \prod\limits_{q \in D, q\ne p} \cF_q \stackrel{\iota}{\lrto}  \da_D(\ff)/ \ff_{\eta}  \stackrel{\psi}{\lrto}  \left( \cF_p  \otimes_{\hat{\oo}_p }  K_p \right) / A_U(\ff) \lrto 0  \, \mbox{,}
\end{equation}
where the group $\ff_{\eta}$ is diagonally embedded into the group $\da_D(\ff)$, the map $\iota $ is induced by the natural embedding
$\prod\limits_{q \in D, q\ne p} \cF_q  \hookrightarrow \da_D(\ff)$, and the map $\psi$ will be defined below.

We consider the adelic complex for the sheaf $j_* j^* \ff$ that calculates the cohomology of the sheaf $j_*j^* \ff$ on the curve $D$:
$$
\ff_{\eta}  \oplus \left(\prod\limits_{q \in D, q\ne p} \cF_q \right) \oplus \left(\cF_p  \otimes_{\hat{\oo}_p }  K_p \right)  \lrto  \da_D(\ff) \, \mbox{.}
$$
Since $U$ is affine, we obtain $H^1(D, j_* \oo_U) =0$. Hence we have
\begin{equation}  \label{dens}
\da_D(\ff)= \ff_{\eta} + \left(\prod\limits_{q \in D, q\ne p} \cF_q \right) + \left(\cF_p \otimes_{\hat{\oo}_p }  K_p \right)   \, \mbox{.}
\end{equation}

Now we construct the map $\psi$ in sequence~\eqref{ex-one-adeles} in the following way. Let $x $ be an element from the group $\da_D(\ff)/ \ff_{\eta}$,
and an element $\tilde{x} \in \da_D(\ff)$ be any lift of $x$.
Then by formula~\eqref{dens} there is an element $f $ from the group $\ff_{\eta}$ such that $f + \tilde{x}$ belongs to the subgroup
$$
\left(\prod\limits_{q \in D, q\ne p} \cF_q \right) \oplus \left(\cF_p \otimes_{\hat{\oo}_p }  K_p \right)  \, \mbox{.}
$$
We define
$$
\psi(x) \eqdef s \cdot {\rm pr}_p (f + \tilde{x})  \, \mbox{,}
$$
where the map ${\rm pr}_p$ is the projection from the group $\da_D(\ff)$ to the group $\cF_p \otimes_{\hat{\oo}_p }  K_p$, and the map $s$
is the natural map $\cF_p \otimes_{\hat{\oo}_p }  K_p  \lrto \left( \cF_p  \otimes_{\hat{\oo}_p }  K_p \right) / A_U(\ff)$.
It is clear that the map $\psi$
is well-defined, since for any other choices $\tilde{x}'$ and $f'$ we have that the element $\tilde{x} -\tilde{x}' + f - f' $ belongs to the subgroup
$$
\ff_{\eta} \cap \left( \left(\prod\limits_{q \in D, q\ne p} \cF_q \right) \oplus \left(\cF_p  \otimes_{\hat{\oo}_p }  K_p \right) \right)
= A_U(\ff) , \mbox{.}
$$

By construction, it is clear that the map $\iota$ is an embedding, the map $\psi$ is a surjection, and $\psi \cdot \iota =0 $. Besides,
we have $\Ker \psi \subset \Image \iota $, since if $\psi(x)=0$, then there is an element $f \in \ff_{\eta}$ such that
$$
f + \tilde{x}  \in \left( \prod\limits_{q \in D, q\ne p} \cF_q \right) +  \left( \ff_{\eta} \cap \prod\limits_{q \in D, q\ne p} \cF_q \right)   \, \mbox{,}
$$
and hence $\tilde{x}  \in \left(\prod\limits_{q \in D, q\ne p} \cF_q \right) + \ff_{\eta} $. Therefore sequence~\eqref{ex-one-adeles} is exact.

\begin{example}  \label{simpl-ex}
{\em
If $\ff = \oo(E)$ for a divisor $E$ on the curve $D$, then sequence~\eqref{ex-one-adeles} is
$$
0 \lrto  \prod\limits_{q \in D, q\ne p} \hat{\oo}_q (E) \stackrel{\iota}{\lrto}  \da_{k(D)}/ k(D)  \stackrel{\psi}{\lrto}    K_p  / A_U(E)  \lrto 0  \, \mbox{,}
$$
where $k(D)$ is the field of rational functions on the curve $D$, and  $A_U(E)= A_U(\oo(E))$.}
\end{example}

\begin{nt} {\em
Clearly, exact sequence~\eqref{ex-one-adeles}  is evidently generalized to the case of several points $p_1, \ldots, p_n$ on the curve $D$ (instead of one point $p$) and an affine open subscheme $U= D \setminus \{p_1, \ldots, p_n  \}$.
}
\end{nt}

\section{Case of two-dimensional schemes}
The goal of this and the next section is
to consider the case of   quotient groups of the  Parshin-Beilinson adelic groups on  two-dimensional schemes (on algebraic surfaces in the next section).

\subsection{Adelic formalism}  \label{ad_form}
Let $X$ be an excellent Noetherian (e.g., finite type over a field $k$ or over $\Z$) normal integral  two-dimensional separated scheme $X$.
For any closed point $x \in X$ let $\hat{\oo}_x$ be a completion of the local ring $\oo_x$ of the point $x$ with respect to the maximal ideal.
Then $\hat{\oo}_x$ is again an integrally closed  domain (see, e.g.,~\cite[ch.~13, \S~33, th.~79]{M}). Let $K_x$ be a localization of the ring $\hat{\oo}_x$ with respect to the multiplicative system
$\oo_x  \setminus 0$. For any one-dimensional integral closed subscheme $D$ of $X$ let $K_D$ be a completion of the field of rational functions on $X$ with respect to the discrete valuation given by $D$.

Let $x \in D$ be any pair such that  $x $ is a closed point on $X$ and $D$ is a one-dimensional integral closed subscheme of $X$.
Let $\rho_i$, where $1 \le i \le l$, be all height one prime ideals of the ring $\hat{\oo}_x$ which contain an ideal $\rho_D \hat{\oo}_x$, where the prime ideal $\rho_D$ of the ring $\oo_x$ defines the subscheme $D \mid_{\Spec \oo_x}$.
We define
$$
K_{x,D} = \prod_{1 \le i \le l} K_i \mbox{,}
$$
where $K_i$ is a two-dimensional local field obtained as the completion of the field $\Frac \hat{\oo}_{x}$ with respect to the discrete valuation given by the prime ideal $\rho_i$. Similarly, we define
$$
\oo_{K_{x,D}}= \prod_{1 \le i \le 1} \oo_{K_i}  \, \mbox{,}
$$
where $\oo_{K_i}$ is the discrete valuation ring of the field $K_i$.

For any quasicoherent sheaf on $X$ there is its adelic group, see its definition, for example,  in~\cite{H, Osi}.
We will use the following notation. For any locally free sheaf $\ff$ on $X$ let $\hat\ff_{q}$ be the completion of the stalk $\ff_q$  of the sheaf $\ff$ at a  point $q$ of $X$.
For a one-dimensional integral closed subscheme $D$ of $X$ by $\hat{\ff}_D$ we mean the completion of the stalk of $\ff$ at the generic point of $D$, i.e. at  a non-closed point on $X$
whose closure coincides with $D$.
We consider the adelic group of the sheaf $\ff$:
$$
\da_X(\ff) = {\prod_{x \in D}}' K_{x,D}(\ff) \;  \subset \;  {\prod_{x \in D}} K_{x,D}(\ff)  \; \mbox{,}
$$
where $K_{x,D}(\ff)= \hat{\ff}_{x}  \otimes_{\hat{\oo}_x} K_{x,D}$, and $x \in D$ run over all pairs as above. We define also
$\oo_{K_{x,D}}(\ff)= \hat{\ff_{x}}  \otimes_{\hat{\oo}_x} \oo_{K_{x,D}}$.

Similarly, there are subgroups  of  $\da_X(\ff)$:
\begin{gather}
\da_{X,01}(\ff)=  \da_X(\ff)  \cap \prod_D K_D(\ff)   \, \mbox{,} \qquad
\da_{X, 02}(\ff) = \da_X(\ff)  \cap \prod_x K_x(\ff)    \, \mbox{,} \notag \\  \label{third}
\da_{X, 12}(\ff) = \da_X(\ff)  \cap \prod_{x \in D} \oo_{K_{x,D}}(\ff)  =
\mathop{{\prod}'}_{x \in D} \oo_{K_{x,D}}(\ff)
\, \mbox{,}
\end{gather}
where
the intersections are taken in  ${\prod\limits_{x \in D}} K_{x,D}(\ff)$, we consider the diagonal embeddings of $\prod\limits_D$ and $\prod\limits_x$
into $\prod\limits_{x \in D}$,  and $K_D(\ff) = \hat{\ff}_D \otimes_{\hat{\oo}_D} K_D $,
$K_x(\ff) = \hat{\ff}_x  \otimes_{\hat{\oo}_x}  K_x$.

Besides, for any subset $\Delta$ of the set of all pairs $x \in D$ as above, we will use notation ${\mathop{\prod}\limits_{\Delta}}'$ for the intersection of $\mathop{\prod}\limits_{\Delta} $
of the same factors with $\da_X(\ff)$ inside of $\mathop{\prod}\limits_{x \in D} K_{x,D}(\ff)$ (compare with formula~\eqref{third}).

\subsection{Adelic quotient group}
\subsubsection{Construction of a surjective map~$\varphi$} \label{sur_map}
We keep the notation from section~\ref{ad_form}.

Let $C$ be a reduced one-dimensional closed equidimensional subscheme
of the scheme~$X$. We consider $C= \bigcup\limits_{1 \le i \le w} C_i$, where $C_i$ are integral one-dimensional closed subschemes of~$X$.
We suppose that $U = X \setminus C$ is an affine scheme. Let $j : U \hookrightarrow X$  be the corresponding embedding.

Let $\ff$ be a locally free sheaf  on $X$.
We note that
$$
H^i(X, j_* j^* \ff ) = H^i (U,  j^* \ff) =0 \qquad \mbox{for} \qquad i \ge 1  \, \mbox{.}
$$

We note also that $j_* j^* \ff =  \ff \otimes_{\oo_X} j_* \oo_U$. Since we supposed that $X$ is an intergal and normal scheme (in particular, the singular locus of $X$ is a finite number of closed points), we have
\begin{equation}   \label{O_n}
j_* \oo_U = \varinjlim_{n} \oo_X(nC)   \, \mbox{,}
\end{equation}
where $\oo_X(nC)$ is a torsion free reflexive coherent subsheaf of the constant sheaf of the field of rational functions on $X$, and this subsheaf consists of elements of $j_* \oo_U$ which have discrete valuations given by subschemes $C_i$ (where $1 \le i \le w$) greater or equal to~$-n$.
(This follows, e.g., from~\cite[\S~3]{BD}.)

From $H^2(X, j_* j^* \ff) =0 $  and the adelic complex for the sheaf $j_* j^* \ff $  we obtain
\begin{equation}  \label{eq_ad}
\da_{X} (j_* j^* \ff)) = \da_{X, 12}(j_* j^* \ff)  + \da_{X, 01}(j_* j^* \ff) + \da_{X, 02}(j_* j^* \ff) \, \mbox{.}
\end{equation}
We have
$$
\da_{X}(j_* j^* \ff) = \da_X(\ff) \, \mbox{,}  \qquad \da_{X, 01}(j_* j^* \ff) = \da_{X, 01}(\ff) \, \mbox{,} \qquad \da_{X, 02}(j_* j^* \ff)= \da_{X, 02}(\ff) \, \mbox{.}
$$
Since for any affine open subscheme $V= \Spec A$ of the scheme $X$ we have
(the proof is analogous to the proofs of~\cite[Prop.~1(5)]{Osi1} and of~\cite[Prop.~4(2)]{Osi1})
$$\da_{V,12}((j_* j^* \ff) \mid_V)= \da_{V,12}(\oo_V) \otimes_A
H^0 (V, (j_* j^* \ff) \mid_V)   \, \mbox{,}$$
we obtain
\begin{equation}  \label{ad12}
\da_{X, 12} (j_* j^* \ff) = \mathop{{\prod}'}_{x \in D, D \nsubset C} \oo_{K_{x,D}}(\ff)  \oplus  \prod_{1 \le i \le w}  \mathop{{\prod}'}_{x \in C_i} K_{x,C_i}(\ff)  \mbox{. }
\end{equation}
We denote
\begin{equation}  \label{dau}
\da_{X, 12}^U (\ff)= {\mathop{{\prod}'}\limits_{x \in D, D \nsubset C}} \oo_{K_{x,D}}(\ff) \, \mbox{.}
\end{equation}

Hence and from formula~\eqref{eq_ad} we obtain
\begin{equation}  \label{h2}
\da_{X}(\ff) = \da_{X, 01}(\ff) + \da_{X, 02}(\ff) + \left( \da_{X, 12}^U (\ff)  \oplus \prod_{1 \le i \le w}  \mathop{{\prod}'}_{x \in C_i} K_{x,C_i}(\ff)  \right) \, \mbox{.}
\end{equation}

\bigskip

Now we want to construct the map $\varphi$:
\begin{gather}
\label{varphi}
\da_X(\ff)/ \left( \da_{X, 01}(\ff) + \da_{X, 02}(\ff)      \right)  \; \stackrel{\varphi}{\lrto}  \;
\left( \prod_{1 \le i \le w}   \mathop{{\prod}'}_{x \in C_i} K_{x,C_i}(\ff) \right) / \Upsilon(\ff)    \, \mbox{,}  \\
\mbox{where the group}
 \qquad \qquad \qquad \qquad \qquad \qquad \qquad \qquad \qquad \qquad \qquad \qquad \qquad \qquad \qquad \qquad \qquad \qquad \quad
 \nonumber
\\ \label{pro}
 \Upsilon(\ff) =  \pro\nolimits_C \left(  \left( \da_{X, 01}(\ff)  + \da_{X, 02}(\ff)    \right)  \cap      \left(
    \da_{X, 12}^U (\ff) \oplus    \prod_{1 \le i \le w}  \mathop{{\prod}'}_{x \in C_i} K_{x,C_i}(\ff)
\right)             \right) \mbox{.}
\end{gather}
Here the map $\pro_C$ is the projection from the group $\da_X(\ff)$ to the group
$\prod\limits_{1 \le i \le w}  \mathop{{\prod}'}\limits_{x \in C_i} K_{x,C_i}(\ff)$.

We define the map $\varphi$ as follows:
$$
\varphi(x) \eqdef s \cdot \pro\nolimits_C (\tilde{x} + g)  \, \mbox{,}
$$
where $s$ is the natural map from the group $\prod\limits_{1 \le i \le w}  \mathop{{\prod}'}\limits_{x \in C_i} K_{x,C_i}(\ff)$  to the group \linebreak
$\prod\limits_{1 \le i \le w}  \mathop{{\prod}'}\limits_{x \in C_i} K_{x,C_i}(\ff)
/ \Upsilon$,
an element
$\tilde{x}$ is any lift of an element $x$ from the group $\da_X(\ff)/ \left( \da_{X, 01}(\ff) + \da_{X, 02}(\ff) \right)$ to the group $\da_X(\ff)$, and an element
$$g  \, \in \, \left( \da_{X, 01}(\ff) + \da_{X, 02}(\ff) \right) $$
 is chosen with the property
\begin{equation} \label{lift}
\tilde{x}  + g  \, \in \,  {\da_{X, 12}^U (\ff) \oplus \prod\limits_{1 \le i \le w}  \mathop{{\prod}'}\limits_{x \in C_i} K_{x,C_i}(\ff)} \, \mbox{.}
\end{equation}
(Such an element $g$ exists by formula~\eqref{h2}.)

The map $\varphi$  is well-defined, since if an element $\tilde{x}'$ is another lift and $g'$ is another choice of corresponding elements, then
we have
$$
\tilde{x}' + g' - \tilde{x} - g \; \in \;
  \left( \da_{X, 01}(\ff)  + \da_{X, 02}(\ff)    \right)  \cap      \left(
    \da_{X, 12}^U (\ff) \oplus    \prod_{1 \le i \le w}  \mathop{{\prod}'}_{x \in C_i} K_{x,C_i}(\ff)
\right)  \, \mbox{.}
$$

Besides, from the construction it is clear that the map $\varphi$ is surjective.

\begin{lemma}  \label{lem1}
There is an equality of subgroups of the group $\da_{X}(\ff)$:
\begin{multline*}
\da_{X,12}(j_* j^* \ff)  \cap \left( \da_{X, 01}(\ff)  + \da_{X, 02}(\ff)     \right) = \\
= \left( \da_{X, 12}(j_* j^* \ff)  \cap \da_{X, 01}(\ff)                 \right)
+
\left(
\da_{X, 12}(j_* j^* \ff)  \cap \da_{X, 02}(\ff)
\right)   \, \mbox{.}
\end{multline*}
\end{lemma}
\begin{proof}
 We  denote the sheaf ${\mathcal H} = j_*j^* \ff$ on $X$ and consider a diagram with exact columns:
$$
\xymatrix{
\da_{X, 0}({\mathcal H})  \ar@{^{(}->}[d] \ar[r]   &
\da_{X, 01}({\mathcal H}) \cap \da_{X, 02}({\mathcal H})  \ar[r] \ar@{^{(}->}[d] & 0 \ar@{^{(}->}[d] \\
\da_{X, 0}(\mh)  \oplus \da_{X, 1}(\mh)  \oplus \da_{X,2}(\mh)  \ar[r]  \ar@{->>}[d] & \da_{X, 01}(\mh)  \oplus \da_{X, 02}(\mh) \oplus \da_{X,12}(\mh)
\ar[r]  \ar@{->>}[d]  & \da_{X}(\mh)   \ar@{->>}[d] \\
\da_{X,1}({\mathcal H})  \oplus \da_{X,2}({\mathcal H})   \ar[r] & \left( \da_{X,01}({\mathcal H}) + \da_{X, 02}({\mathcal H})    \right) \oplus \da_{X, 12}({\mathcal H}) \ar[r] & \da_{X}({\mathcal H})
\, \mbox{,}}
$$
where the middle row is the adelic complex  for the sheaf $\mh$ on $X$. This adelic complex calculates the cohomology of $\mh$ on $X$.
 Therefore the statement of the lemma follows from the long exact sequence constructed by this diagram and from the facts:
$H^1(X, {\mathcal H})=0 $ and
\begin{equation}  \label{intersect}
\da_{X,1}({\mathcal H}) = \da_{X, 01}({\mathcal H}) \cap \da_{X, 12}({\mathcal H}) \, \mbox{,}
\qquad \qquad
\da_{X,2}({\mathcal H})= \da_{X, 02}({\mathcal H} )  \cap \da_{X, 12}({\mathcal H}) \, \mbox{.}
\end{equation}
To prove formulas~\eqref{intersect}, we note that it is enough to consider a sheaf $\ff \otimes_{\oo_X} \oo_X(nC) $ (where $n \ge 0$) instead of the sheaf $\mh$, because
$\mh =  \varinjlim\limits _{n} \left(\ff \otimes_{\oo_X} \oo_X(nC) \right)$, and adelic factors commute with direct limits. Now the first equality in~\eqref{intersect} is true since it is easy to see for a locally free sheaf and  the sheaf $\ff \otimes_{\oo_X} \oo_{X}(nC)$ is invertible except for the finite number of singular points of $X$ which do not impact on the equality. The second equality is again easy to see for an open subscheme of $X$ where the sheaf $\ff \otimes_{\oo_X} \oo_X(nC)$ is a locally free sheaf, because $X$ is a normal scheme. Therefore it is enough to check this equality locally for a singular (closed) point $x$ of $X$, where it follows from an equality of $\hat{\oo}_x$-submodules of $\Frac \hat{\oo}_x$:
\begin{equation}  \label{compl}
\oo_x(nC) \otimes_{\oo_x} \hat{\oo}_x = \hat{\oo}_x(nC \mid_{\Spec(\hat{\oo}_x)} ) \, \mbox{.}
\end{equation}
Equality~\eqref{compl} is true, because the completion of a maximal Cohen-Macaulay $\oo_x$-module (which is the same as reflexive module in this case) is a maximal Cohen-Macaulay $\hat{\oo}_x$-module (i.e. a reflexive module).
\end{proof}

For any pair $x \in D$ on $X$ (as in Section~\ref{ad_form}) we have a canonical embedding of groups
$$
\xymatrix{
p_{x,D, \ff} \; : \; K_x(\ff) \; \ar@{^{(}->}[r]   & \;  K_{x,D}(\ff) \, \mbox{.}
}
$$
For any closed point $x \in X $ we define a subgroup $B_{x,C}(\ff)$ of the group $K_{x}(\ff)$ in the following way:
$$
B_{x,C}(\ff) \eqdef \bigcap_{D \ni x, D \nsubset C} p_{x,D, \ff}^{-1} \left( p_{x,D, \ff}(K_x(\ff))  \cap \oo_{K_{x,D}}(\ff) \right)
$$
\begin{nt}  \label{another} {\em
From the proof of Lemma~\ref{lem1} we have that
$$B_{x,C} (\ff)  = \hat{\oo}_x \otimes_{\oo_x}  (j_* j^* \ff)_x = \hat{\ff}_x  \otimes_{\oo_x} (j_* j^* \oo_X)_x \, \mbox{.}$$
}
\end{nt}
\begin{nt} \em
In~\cite[\S~14]{OsipPar2} we used notation $B_x$ instead of $B_{x,C}$ (i.e., without indicating on the subscheme $C$).
\end{nt}

Now from Lemma~\ref{lem1} and formula~\eqref{ad12} we immediately obtain
\begin{multline}
\da_{X, 12}(\ff) \cap \left( \da_{X, 01}(\ff)  + \da_{X, 02}(\ff) \right)=   \\
 =
\left( \prod_{D \subset X, D \nsubset C}  \hat{\ff}_{D}  \oplus \prod_{1 \le i \le w} K_{C_i}(\ff)  \right)
+
\left(
\prod_{x \in U} \hat{\ff}_x   \oplus \mathop{{\prod}'}_{x \in C} B_{x,C}(\ff)
\right)   \, \mbox{.}    \label{h1}
\end{multline}

Therefore the projection map (see formula~\eqref{pro})  gives
$$
\pro\nolimits_C \left(  \da_{X, 12}(\ff)  \cap \left( \da_{X, 01}(\ff)  + \da_{X,02}(\ff)  \right)   \right)
= \prod_{1 \le i \le w} K_{C_i}(\ff)  + \mathop{{\prod}'}_{x\in C} B_{x,C}(\ff)  \, \mbox{.}
$$

Thus, bearing in mind formula~\eqref{varphi}, we have that the map $\varphi$ is the following surjective map:
$$
\hspace{-0.1cm}
\da_X(\ff)/ \left( \da_{X, 01}(\ff) + \da_{X, 02}(\ff)      \right)   \stackrel{\varphi}{\lrto}
\left( \prod_{ 1 \le i \le w}  \mathop{{\prod}'}_{x \in C_i}   K_{x, C_i}(\ff)  \right)
/
\left( \prod_{1 \le i \le w} K_{C_i}(\ff)   + \mathop{{\prod}'}_{x \in C} B_{x,C}(\ff)                     \right)
 \mbox{.}
$$

\subsubsection{Calculation of the kernel}  \label{calc_ker}
We consider now a natural  map
$$
\mathop{{\prod}'}_{x \in D, D \nsubset C} \oo_{K_{x,D}}(\ff) \;  \stackrel{\phi}{\lrto}  \; \da_{X}(\ff) / \left(  \da_{X, 01}(\ff)  + \da_{X, 02}(\ff)   \right)  \, \mbox{.}
$$
It is clear that $\Image \phi \subset \Ker \varphi$. (Indeed, for the construction of $\varphi$ we take $g=0$ and the lift $\tilde{x}$ which comes from the image of $\phi$.)

Let us show that $\Ker \varphi \subset \Image \phi $.  Let $\varphi(x)= 0 $ for an element $x $
 from the group $
 \da_{X}(\ff) / \left(  \da_{X, 01}(\ff)  + \da_{X, 02}(\ff) \right)$.
Then, by construction, there is an element ${g }$ from the group $  {\da_{X, 01}(\ff)  + \da_{X, 02}(\ff)}$ such that
$$
\tilde{x} + g \; \in \;
\mathop{{\prod}'}_{x \in D, D \nsubset C} \oo_{K_{x,D}}(\ff) +
\left( \da_{X, 01}(\ff)  + \da_{X,02}(\ff) \right) \, \mbox{,}
$$
where ``$\in$''  follows from formulas~\eqref{lift} and~\eqref{pro}.
Therefore we have
$$
\tilde{x}  \; \in \;    \mathop{{\prod}'}_{x \in D, D \nsubset C} \oo_{K_{x,D}}(\ff)  +\left( \da_{X, 01}(\ff)  + \da_{X,02}(\ff)  \right) \, \mbox{.}
$$

Thus, we obtain $\Image \phi = \Ker \varphi$. Therefore for an explicit description of  the group  $\da_{X}(\ff) / \left(  \da_{X, 01}(\ff)  + \da_{X, 02}(\ff)   \right) $,  we have to calculate the group $\Ker \phi$.

It is clear that
$$
\Ker \phi =    \left(  \da_{X, 01}(\ff)  + \da_{X, 02}(\ff)   \right) \,
\cap
\mathop{{\prod}'}_{x \in D, D \nsubset C} \oo_{K_{x,D}}(\ff)
\, \mbox{.}
$$
Let us calculate $\Ker \phi$ more explicitly.

For any pair $x \in D$ on $X$ (as in Section~\ref{ad_form}) we have canonical embeddings of groups
$$
\xymatrix{
q_{x,D, \ff} \; : \; K_D(\ff)  \; \ar@{^{(}->}[r]  & \; K_{x,D}(\ff)  \,  \mbox{.}
}
$$

\medskip

Now we {\em define} a subgroup $A_C(\ff) \subset \prod\limits_{1 \le i \le w} K_{C_i}(\ff)$\footnote{We recall that $C = \bigcup\limits_{1 \le i \le w} C_i$, where $C_i$ is an integral one-dimensional subscheme of $X$.} as the image of the  projection of the group
$\Ker \Xi$ to the group   $ \prod\limits_{1 \le i \le w}  K_{C_i}(\ff)$, where the map
\begin{equation}  \label{Xi}
\Xi \, : \,  \prod\limits_{1 \le i \le w} K_{C_i}(\ff)  \oplus \prod_{x \in C} B_{x,C}(\ff)  \lrto   \prod\limits_{1 \le i \le w}
\prod_{x \in C_i}  K_{x,C_i}(\ff)       \, \mbox{,}
\end{equation}
and   $\Xi (z \oplus v) = \prod\limits_{1 \le i \le w}
\prod\limits_{x \in C_i} q_{x, C_i, \ff}(z)   -  \prod\limits_{1 \le i \le w}
\prod\limits_{x \in C_i} p_{x, C_i, \ff}(v) $ for elements $z \in \prod\limits_{1 \le i \le w} K_{C_i}(\ff) $
and $v \in \prod\limits_{x \in C} B_{x,C}(\ff) $.

We note that if  $w=1$, i.e. $C = C_1$, then
$$
A_C(\ff) = \bigcap_{x \in C}
q_{x, C, \ff}^{-1} \left(   q_{x,C,\ff}(K_{C}(\ff)) \cap p_{x,C, \ff}(B_{x,C}) \right)
$$

We have a natural embedding $\tau$:
\begin{equation}  \label{tau}
\tau \; : \;
A_C(\ff)
\hookrightarrow  \mathop{{\prod}'}_{x \in C} B_{x,C}(\ff)
\hookrightarrow \mathop{{\prod}'}_{x \in X}  K_x(\ff)  \hookrightarrow \da_X(\ff)  \, \mbox{,}
\end{equation}
where the first arrow denotes the map $z \mapsto v$ (see definition of the map $\Xi$ in~\eqref{Xi}).

We have also a natural embedding $\gamma$:
\begin{equation}  \label{gamma}
\gamma \; : \; A_C(\ff)  \hookrightarrow \prod_{1 \le i \le w} K_{C_i}(\ff)  \hookrightarrow \da_X(\ff)  \, \mbox{.}
\end{equation}

We {\em claim} that
\begin{multline}
\Ker \phi =  \left(  \da_{X, 01}(\ff) \,  +  \, \da_{X, 02}(\ff)   \right) \;
\cap  \;
\mathop{{\prod}'}_{x \in D, D \nsubset C} \oo_{K_{x,D}}(\ff)   \;  \supset
\\
\supset  \;
 \prod_{D \subset X, D \nsubset C} \hat{\ff}_D  \,
+   \,
\prod_{x \in U}  \hat{\ff}_x   \,
+  \,
(\tau - \gamma)(A_C(\ff))  \mbox{.}   \label{ker}
\end{multline}
Indeed, it is clear that the first two summands from the last sum belong to the group $\Ker \phi $.
Besides, by construction, we have that
\begin{gather*}
\tau(A_C(\ff))  \subset \da_{X, 02} (\ff)  \; \mbox{,}   \qquad \gamma(A_C(\ff))  \subset \da_{X, 01}(\ff)
\, \mbox{,} \\   \mbox{and} \quad
(\tau - \gamma)(A_C (\ff))  \subset   \mathop{{\prod}'}_{x \in D, D \nsubset C} \oo_{K_{x,D}}(\ff)  \, \mbox{.}
\end{gather*}
Therefore, $(\tau - \gamma)(A_C(\ff))  \subset \Ker \phi$.

On the other hand (recall notation from formula~\eqref{dau}),
$$
\Ker \phi = \da_{X, 12}^U (\ff)  \cap \left( \da_{X, 01}(\ff)  + \da_{X, 02}(\ff)     \right)  \;
\subset  \;
\da_{X, 12}(\ff)  \cap \left( \da_{X, 01}(\ff)   + \da_{X, 02}(\ff)  \right)  \mbox{.}
 $$

Now, bearing in mind formula~\eqref{h1}, and using that the projection of the group $\Ker \phi$ to any group $K_{x,C_i}(\ff)$ is zero,
we obtain that in the group $\Ker \phi$ elements from the group $\prod\limits_{1 \le i \le w} K_{C_i}(\ff)$ should be ``compensated'' by elements from the group
$\mathop{\prod'}\limits_{x \in C} B_{x,C}(\ff)$ to obtain zero. Hence and from formula~\eqref{ker} we obtain that
\begin{equation}  \label{kernel}
\Ker \phi = \prod_{D \subset X, D \nsubset C} \hat{\ff}_D  \,
+   \,
\prod_{x \in U}  \hat{\ff}_x   \,
+  \,
(\tau - \gamma)(A_C(\ff))  \, \mbox{.}
\end{equation}

Thus, we have proved a theorem.
\begin{Th}  \label{ex-seq}
There is the following exact sequence
\begin{multline*}
0 \lrto
\frac{\mathop{{\prod}'}\limits_{x \in D, D \nsubset C} \oo_{K_{x,D}}(\ff)}{ \prod\limits_{D \subset X, D \nsubset C} \hat{\ff}_D
+
\prod\limits_{x \in U}  \hat{\ff}_x
+
(\tau - \gamma)(A_C(\ff)) }  \stackrel{\phi}{\lrto}
\frac{\da_X(\ff)}{  \da_{X, 01}(\ff) + \da_{X, 02}(\ff)    }   \stackrel{\varphi}{\lrto}  \\ \stackrel{\varphi}{\lrto}
\frac{ \prod\limits_{ 1 \le i \le w }  \mathop{{\prod}'}\limits_{x \in C_i}   K_{x, C_i}(\ff))}{
 \prod\limits_{1 \le i \le w} K_{C_i}(\ff)   + \mathop{{\prod}'}\limits_{x \in C} B_{x,C}(\ff)   }
 \lrto 0 \, \mbox{.}
\end{multline*}
\end{Th}

\bigskip

Now we calculate explicitly the group $A_C(\ff)$.

\begin{prop}  \label{A_C}
There is an isomorphism
\begin{equation}  \label{prop1}
A_C(\ff)  \simeq \varinjlim\limits _{n} \varprojlim\limits_{m < n}  H^0 \left( X, \ff \otimes_{\oo_X} \left( \oo_X(nC)/   \oo_X(mC) \right)  \right)  \,
\mbox{,}
\end{equation}
which comes from canonical embeddings  (with the same image) of groups from the left and right hand sides of formula~\eqref{prop1} into the group  $\prod\limits_{1 \le i \le w} K_{C_i}(\ff)$.
\end{prop}
\begin{proof}
Let $J_C$ be the ideal sheaf of the subscheme $C$ on $X$. For integers ${n > m}$ we consider a $1$-dimensional closed subscheme $Y_{n-m}=(C, \oo_X/J_C^{n-m}) \subset X$
with the topological space $C$ and the structure sheaf $\oo_X/J_C^{n-m}$. The sheaf
$${\ff_{n,m} = \ff \otimes_{\oo_X} \left( \oo_X(nC)/   \oo_X(mC) \right) }$$
is a coherent sheaf on the scheme $Y_{n-m}$. Now the proof follows from the calculation of the group $H^0 \left( X, \ff_{n,m} \right) = H^0(Y_{n-m}, \ff_{n,m})$ via the adelic complex for the sheaf $\ff_{n,m} $ on the
$1$-dimensional scheme $Y_{n-m}$ (see also formula~\eqref{O_n} and Remark~\ref{another}) and the passing to injective limit on $n$ and projective limit on $m$.
The final complex after the passing to injective limit on $n$ and projective limit on $m$ looks as:
$$
\prod\limits_{1 \le i \le w} K_{C_i}(\ff)  \oplus \mathop{{\prod}'}_{x \in C} B_{x,C}(\ff)  \lrto   \prod\limits_{1 \le i \le w}
\mathop{{\prod}'}_{x \in C_i}  K_{x,C_i}(\ff) \, \mbox{.}
$$
\end{proof}

\section{Case of projective algebraic surface} \label{proj-sur}
Now we restrict ourself to the case of a projective algebraic normal irreducible surface $X$ over a field $k$.
We recall that $C = \bigcup\limits_{1 \le i \le w} C_i$, where $C_i$ are irreducible closed curves on $X$.  We suppose that
$\tilde{C}= \bigoplus\limits_{1 \le i \le w} m_iC_i $  (where $m_i \ge 1$ are certain integers) is   an ample divisor on $X$, i.e. $\oo_X(l\tilde{C})$ is a very ample invertible sheaf on $X$ for certain integer $l \ge 0$.

\begin{prop}   \label{A_C-proj}
There is an isomorphism
$$
H^0(X, j_* j^* \ff) \simeq
\varinjlim\limits _{n} \varprojlim\limits_{m < n}  H^0 \left( X, \ff \otimes_{\oo_X} \left( \oo_X(nC)/   \oo_X(mC)  \right) \right)
 \, \mbox{,}
$$
which comes from the composition of maps of sheaves
$$
j_*j^* \ff \simeq \varinjlim\limits _{n} \ff  \otimes_{\oo_X}  \oo_X(nC)  \lrto
\varinjlim\limits _{n} \varprojlim\limits_{m < n}   \ff \otimes_{\oo_X} \left( \oo_X(nC)/   \oo_X(mC)  \right)  \, \mbox{.}
$$
\end{prop}
\begin{proof}
We denote a  divisor $C'= l \tilde{C}$ such that $\oo_X(C')$ is a very ample invertible sheaf.
It is clear that
$$
\varinjlim\limits _{n} \varprojlim\limits_{m < n}  H^0 \left( X, \ff \otimes_{\oo_X} \left( \oo_X(nC)/   \oo_X(mC)  \right) \right) \simeq
\varinjlim\limits _{r} \varprojlim\limits_{s < r}  H^0 \left( X, \ff \otimes_{\oo_X} \left( \oo_X(rC')/   \oo_X(sC')  \right) \right)  \mbox{,}
$$
because of isomorphism of the corresponding IndPro systems.

A normal surface is a Cohen-Macaulay scheme. Since $C'$ is a very ample divisor, by~\cite[ch.~III, Th.~7.6]{Ha}
we have
$$
H^i(X, \ff(sC'))  =0 \qquad \mbox{when} \qquad i \le 1  \qquad \mbox{and}  \qquad s \ll 0  \, \mbox{.}
$$
Now  from apllication of these equalities to  an exact sequence of sheaves on $X$:
$$
0 \lrto  \ff(sC)  \lrto \ff(rC)   \lrto \ff \otimes_{\oo_X}  \left( \oo_X(rC) / \oo_X(sC)    \right)  \lrto 0 \, \mbox{.}
$$
we obtain an isomorphism
$$
\varinjlim\limits _{r} \varprojlim\limits_{s < r}  H^0 \left( X, \ff \otimes_{\oo_X} \left( \oo_X(rC')/   \oo_X(sC')  \right) \right)
\simeq
\varinjlim\limits _{r}   H^0 \left( X, \ff \otimes_{\oo_X} \oo_X(rC') \right) \, \mbox{.}
$$
Finally we have
$$
\varinjlim\limits _{r}   H^0 \left( X, \ff \otimes_{\oo_X} \oo_X(rC') \right)  \simeq H^0(X, j_* j^* \ff) \, \mbox{.}
$$
\end{proof}

From Propositions~\ref{A_C} and~\ref{A_C-proj} it follows that $$A_C(\ff)  = H^0(X, j_* j^* \ff) \, \subset \, \ff_{\eta}  \, \mbox{,}$$
where $\eta$ is the generic point of $X$. Therefore for an element $a$ from the group $A_C(\ff)$ we have
$$
(\tau - \gamma)(a) = I(a) - J(a)  \, \mbox{,}
$$
where $I(a)$ is the image of the element $a$ into $\prod\limits_{D \subset X, D \nsubset C} \hat{\ff}_{D}$  under the diagonal embedding,
and $J(a)$ is the image of the element $a$ into $\prod\limits_{x \in U}  \hat{\ff}_x$ under the diagonal embedding.\footnote{We recall also that the maps $\tau$ and $\gamma$ were defined by formulas~\eqref{tau} and~\eqref{gamma} correspondingly.} Thus we obtain that
$$
(\tau - \gamma)(A_C(\ff)) \; \subset \; \prod\limits_{D \subset X, D \nsubset C} \hat{\ff}_{D} + \prod\limits_{x \in U}  \hat{\ff}_x  \, \mbox{.}
$$
Therefore from formula~\eqref{kernel} we obtain
\begin{equation}  \label{Kernel-surf}
\Ker \phi = \prod_{D \subset X, D \nsubset C} \hat{\ff}_D  \,
+   \,
\prod_{x \in U}  \hat{\ff}_x   \, \mbox{.}
\end{equation}

Thus, from formula~\eqref{Kernel-surf} and Theorem~\ref{ex-seq}
we deduce a theorem.

\begin{Th}  \label{anal-II}
Le $\tilde{C}= \bigoplus\limits_{1 \le i \le w} m_i C_i$  (with certain integers $m_i \ge 1$) be an ample  divisor on a normal projective irreducible algebraic surface $X$ over a field $k$, where $C_i$ are irreducible curves. Let an open subscheme $U = X \setminus C$, where $ C= \bigcup\limits_{1 \le i \le w} C_i$. For any locally free sheaf $\ff$ on $X$ there is the following exact sequence.
 \begin{multline}  \label{seq_th}
0 \lrto
\frac{\mathop{{\prod}'}\limits_{x \in D, D \nsubset C} \oo_{K_{x,D}}(\ff)}{ \prod\limits_{D \subset X, D \nsubset C} \hat{\ff}_D
+
\prod\limits_{x \in U}  \hat{\ff}_x
 }  \stackrel{\phi}{\lrto}
\frac{\da_X(\ff)}{  \da_{X, 01}(\ff) + \da_{X, 02}(\ff)    }   \stackrel{\varphi}{\lrto}  \\ \stackrel{\varphi}{\lrto}
\frac{ \prod\limits_{ 1 \le i \le w}  \mathop{{\prod}'}\limits_{x \in C_i}   K_{x, C_i}(\ff))}{
 \prod\limits_{1 \le i \le w} K_{C_i}(\ff)   + \mathop{{\prod}'}\limits_{x \in C} B_{x,C}(\ff)   }
 \lrto 0 \, \mbox{.}
\end{multline}
\end{Th}

\begin{nt}{\em \label{gap}
As we mentioned in Introduction, Theorem~\ref{anal-II} was formulated
as Theorem~3 in~\cite{OsipPar2}
for  a smooth projective connected algebraic surface $X$, a sheaf $\ff = \oo_X$, and all $m_i=1$. But the proof in~\cite{OsipPar2} was incorrect: it contained gaps. In particular, an additional term $(\tau - \gamma)(A_C(\ff)) $ from Theorem~\ref{ex-seq} was not considered.
}
\end{nt}

From Theorem~\ref{anal-II} we explicitly see that $k$-vector space $\da_{X}(\ff)/ \left(  \da_{X, 01}(\ff) + \da_{X, 02}(\ff)    \right)$
is a linearly compact  $k$-vector space (or a compact topological space when $k$ is a finite field.) Indeed, the first non-zero term in exact sequence~\eqref{seq_th} can be rewritten as
$$
\frac{\mathop{{\prod}'}\limits_{x \in D, D \nsubset C} \oo_{K_{x,D}}(\ff)}{ \prod\limits_{D \subset X, D \nsubset C} \hat{\ff}_D +
\prod\limits_{x \in U}  \hat{\ff}_x
}
\;
\simeq
\;
\frac{ \prod\limits_{D \subset X, D \nsubset C } \left( \left(\frac{\mathop{{\prod}'}\limits_{x \in D}  \oo_{K_{x,D}}}{ \hat{\oo}_D}  \right) \otimes_{\hat{\oo}_D} \hat{\ff}_D \right)}{ \prod\limits_{x \in U} \hat{\ff}_x  }  \, \mbox{,}
$$
and $k$-vector spaces $\hat{\ff}_x$ for any point $x \in X$  and $\left(\mathop{{\prod}'}\limits_{x \in D}  \oo_{K_{x,D}} \right) / \hat{\oo}_D $
for any irreducible curve  $D \subset X$ are linearly compact $k$-vector spaces, see~\cite[Remark~26]{OsipPar2}. The last non-zero term in exact sequence~\eqref{seq_th} can be rewritten analogously and by the same arguments as in  Theorem~4 from~\cite{OsipPar2}. (On a two-dimensional local field and, more generally, on the group $\da_X$  there is the natural topology of inductive and projective limits, first introduced by A.~N.~Parshin in~\cite{P1}  for the case $\df_q((t_1))  \ldots ((t_n))$, see also its properties, e.g., in~\cite[\S~3.2]{GOMS}.)

 At the end we recall that using these and other  calculations, it was suggested in~\cite[\S~14.3]{OsipPar2} an analogy for the first and last non-zero terms of exact sequence~\eqref{seq_th} in the case of the simplest arithmetic surface  $\mathbb{P}^1_{\sdz}$. In particular, the last non-zero term should be
equal to
$$
\mathbb{R}((t)) / \left( \mathbb{Z}((t)) + \mathbb{R}[t^{-1}]   \right)  \, \mbox{.}
$$
This will be justified by explicit calculations with arithmetic adeles  on  arithmetic surfaces in~\cite{Osip1}. (Arithmetic adeles, i.e. adeles on an arithmetic surface  when  fibres over Archimedean points of the base are taken into account, were introduced in~\cite[Example~11]{OsipPar2}. See also applications in~\cite[\S~4]{Osi2016}.)


\noindent Steklov Mathematical Institute of Russsian Academy of Sciences \\
\noindent  8 Gubkina St., Moscow 119991, Russia \\

\medskip
\noindent {\it E-mail:}  ${d}_{-} osipov@mi.ras.ru$

\end{document}